\documentclass{amsart}

\usepackage[
  hmarginratio={1:1},     
  vmarginratio={1:1},     
  textwidth=15cm,        
  textheight=21cm,
  heightrounded,          
]{geometry}

\usepackage{graphicx}
\usepackage{amsmath,amssymb,amsthm,amsfonts}
\usepackage{amssymb}

\usepackage{dsfont}
\usepackage{color, framed}

\usepackage[colorlinks]{hyperref}
\hypersetup{linkcolor=blue,citecolor=blue,filecolor=black,urlcolor=blue}

\newtheorem{thm}{Theorem}[section]
\newtheorem{theorem}[thm]{Theorem}
\newtheorem{cor}[thm]{Corollary}
\newtheorem{lem}[thm]{Lemma}

\theoremstyle{definition}

\newtheorem{remark}[thm]{Remark}

\numberwithin{equation}{section}

\newcommand{\R}{{\mathbb{R}}}
\newcommand{\N}{{\mathbb{N}}}

\newcommand{\pa}{\partial}

\newcommand{\cS}{\mathcal{S}}
\newcommand{\cC}{\mathcal{C}}
\newcommand{\cO}{\mathcal{O}}

\DeclareMathOperator{\loc}{loc}

\begin{document}

%

\title[Saddle-shaped positive solutions for elliptic systems with bistable nonlinearity]{Saddle-shaped positive solutions for elliptic systems with bistable nonlinearity}

\author{Nicola Soave}


\address{
\hbox{\parbox{5.7in}{\medskip\noindent
Nicola Soave\\
Dipartimento di Matematica, Politecnico di Milano, \\
Piazza Leonardo da Vinci, 32, 20133 Milano (Italy) \\[2pt]
{\em{E-mail address: }}{\tt nicola.soave@gmail.com, nicola.soave@polimi.it}}}}

%
%
%
\thanks{\emph{Acknowledgements:} the author is partially supported by the INDAM-GNAMPA group.}

\begin{abstract}
In this paper we prove the existence of infinitely many saddle-shaped positive solutions for non-cooperative nonlinear elliptic systems with bistable nonlinearities in the phase-separation regime. As an example, we prove that the system 
\[
\begin{cases}
-\Delta u  =u-u^3-\Lambda uv^2 \\
-\Delta v  =v-v^3-\Lambda u^2v \\
u,v > 0 
\end{cases} \qquad \text{in $\R^N$, with $\Lambda>1$,}
\]
has infinitely many saddle-shape solutions in dimension $2$ or higher. This is in sharp contrast with the case $\Lambda \in (0,1]$, for which, on the contrary, only constant solutions exist.
\end{abstract}

\maketitle


\maketitle

\section{introduction}

This paper concerns existence of multiple positive solutions for certain non-cooperative nonlinear elliptic systems with bistable nonlinearities, whose prototype is 
\begin{equation}\label{model}
\begin{cases}
-\Delta u  =u-u^3-\Lambda uv^2 \\
-\Delta v  =v-v^3-\Lambda u^2v \\
u,v > 0 
\end{cases} \qquad \text{in $\R^N$, with $\Lambda>1$.}
\end{equation}
This system arises in the study of domain walls and interface layers for 
two-components Bose-Einstein condensates \cite{ABCP}. Domain walls solutions satisfying asymptotic conditions 
\begin{equation}\label{cond infty}
\begin{cases}(u,v) \to (1,0) & \text{as $x_N \to +\infty$},\\
(u,v) \to (0,1) & \text{as $x_N \to -\infty$},\end{cases},
\end{equation}
in dimension $N=1$ have been carefully studied in \cite{AS, ABCP}, where in particular it is shown the existence of such a solution for every $\Lambda > 1$ \cite{ABCP}, and its uniqueness in the class of solutions with one monotone component \cite{AS}. In fact, uniqueness holds also without such assumption, and even in higher dimension \cite{FSS}; precisely, in \cite{FSS} it is shown that a solution to \eqref{model}-\eqref{cond infty} (with the  limits being uniform in $x'\in \R^{N-1}$) in $\R^N$ with $\Lambda>1$ is necessarily montone in both the components with respect to $x_N$, and $1$-dimensional. The assumption $\Lambda>1$ is natural, since \eqref{model}-\eqref{cond infty} has no solution at all when $\Lambda \in (0,1]$. Indeed, it is proved that \eqref{model} has only constant solutions for both $\Lambda \in (0,1)$ \cite{FSS}, and $\Lambda =1$ \cite{FSS, S}. 

We also refer to \cite{AfFa,AfSo} for recent results regarding a system obtained from \eqref{model} adding in each equation an additional term representing the spin coupling. 

To sum up, up to now it is known that \eqref{model} has only constant solutions for $\Lambda \in (0,1]$, and at least one $1$-dimensional non-constant solution for $\Lambda >1$. Moreover, solutions with uniform as in \eqref{cond infty} are necessarily $1$-dimensional, and unique modulo translations. In this paper we prove the existence of infinitely many geometrically distinct solutions to \eqref{model} in any dimension $N \ge 2$, for any $\Lambda>1$. This result enlightens once more the dichotomy $\Lambda \in (0,1]$ vs. $\Lambda>1$. While for $\Lambda \in (0,1]$ problem \eqref{model} is rigid in itself, and only possesses constant solutions, for $\Lambda>1$ we have multiplicity of non-constant solutions, and rigidity results can be recovered only with some extra assumption, such as \eqref{cond infty}. 

Our result is based upon variational methods, and strongly exploits the symmetry of the problem. Roughly speaking, we shall construct solutions to \eqref{model} such that $u-v$ ``looks like" a sing changing solution of the Allen-Cahn equation $-\Delta w = w-w^3$, with $u \simeq w^+$, and $v \simeq w^-$. The building blocks $w$ in our construction will be both the saddle-type planar solutions (also called ``pizza solutions") \cite{AlCaMo}, and the saddle solutions in $\R^{2m}$ \cite{CaTe, DaFiPe}.

\subsection{Statement of the main results}

We consider the following general version of \eqref{model}:
\begin{equation}\label{system}
\begin{cases}
-\Delta u =f(u)-\Lambda u^p v^{p+1}& \text{in}\quad\mathbb{R}^N  \\
-\Delta v  =f(v)-\Lambda u^{p+1}v^p & \text{in}\quad\mathbb{R}^N  \\
u,v > 0 &  \text{in}\quad\mathbb{R}^N, 
\end{cases}\qquad \text{with $\Lambda > 0$,}
\end{equation}
where $N \ge 2$, $p \ge 1$, and $f$ is of bistable type; more precisely, let $f: \R \to \R$ be a locally Lipschitz continuous and odd nonlinearity. For a value $M>0$, we define the potential
\[
F(t) = \int_t^M f(s)\,ds,
\]
so that $F \in C^{1,1}(\R)$, and $F' = - f$. We suppose that:
\begin{equation}\label{hp bistable}
F \ge 0 = F(\pm M) \quad \text{in $\R$}, \quad \text{and} \quad F>0 \quad \text{in $(-M,M)$}.
\end{equation}
Note that in this case $f(0) = f(\pm M) = 0$. $F$ is often called a double well potential, and $f$ is called bistable nonlinearity. A simple example is $f(t) =t-t^3$. 

With the above notation, we introduce 
\[
W(s,t) = F(s) + F(t) + \frac{\Lambda}{p+1} |s|^{p+1} |t|^{p+1}, \qquad (s,t) \in \R^2.
\]

The first of our main result concerns the existence of infinitely many geometrically distinct solutions for problem \eqref{system} in the plane. We consider polar coordinates $(r,\theta) \in [0,+\infty) \times [0,2\pi)$ in the plane. For any positive integer $k$, we define:
\begin{itemize}
\item[] $R_k$, the rotation of angle $\pi/k$ in counterclockwise sense;
\item[] $R_k^i$, the rotation of angle $i \pi/k$ in counterclockwise sense, with $i=1,\dots, 2k$;
\item[] $\ell_0$, the line of equation $x_2= \tan(\pi/2k) x_1$ in $\R^2$;
\item[] $\ell_i$, the line $R_k^i(\ell_1)$, $i=1,\dots,k-1$;
\item[] $T_i$, the reflection with respect to $\ell_i$;
\item[] $\alpha_k = \tan(\pi/(2k))$;
\item[] $\cS_k$, the open circular sector $\{-\pi/(2k)<\theta<\pi/(2k)\} = \{ \alpha_k x_1 > |x_2|\} \subset \R^2$.
\end{itemize}
\begin{theorem}[Saddle-type solutions in the plane]\label{thm: ex 2d}
Let $p \ge 1$, $f \in C^{0,1}(\R)$ be odd, and suppose that its primitive $F$ satisfies \eqref{hp bistable}. Suppose moreover that:
\begin{equation}\label{cond on critical}
\inf_{s \in [0,M]} W(s,s) > F(0). 
\end{equation}
Then, for every positive integer $k$, there exists a positive solution $(u_k,v_k)$ to system \eqref{system} in $\R^2$ having the following properties:
\begin{itemize}
\item[($i$)] $0 < u_k,v_k<M$ in $\R^2$; 
\item[($ii$)] $v_k  = u_k \circ T_i$ for every $i=1,\dots,k$, and $u_k(x_1,-x_2) = u_k(x_1,x_2)$ in $\R^2$; 
\item[($iii$)] $u_k-v_k>0$ in $\cS_k$. 
\end{itemize}
\end{theorem}

Notice in particular that $\{u_k-v_k=0\} = \bigcup_{i=0}^{k-1} \ell_i$, which implies that $(u_k,v_k) \neq (u_j,v_j)$ if $j \neq k$. Regarding assumption \eqref{cond on critical}, we stress that for any bistable $f$ it is satisfied provided that $\Lambda>0$ is sufficiently large, and can be explicitly checked in several concrete situations. In particular:

\begin{cor}
For every $\Lambda>1$, problem \eqref{model}
has infinitely many geometrically distinct non-constant solutions.
\end{cor}

The corollary follows from the theorem, observing that if $f(s) = s- s^3$ and $p=1$, then
\[
F(s) = \frac{(1-s^2)^2}{4}, \quad W(s,t) = F(s) + F(t) + \frac{\Lambda s^2 t^2}{2};
\]
thus, condition \eqref{cond on critical} is satisfied if and only if
\[
\inf_{s \in [0,1]} W(s,s) = W\left( \frac{1}{\sqrt{1+\Lambda}}, \frac{1}{\sqrt{1+\Lambda}} \right) = \frac{\Lambda}{2(1+\Lambda)}  > \frac14 = F(0),
\]
that is, if and only if $\Lambda>1$. Notice that, if \eqref{cond on critical} is violated, we have have non-existence of non-constant solutions \cite{FSS,S}, and hence \eqref{cond on critical} is sharp in this case.

The proof of Theorem \ref{thm: ex 2d} consists in a $2$ steps procedure. At first, we construct a solution to \eqref{system} in a ball $B_R$ with the desired symmetry properties, combining variational methods with an auxiliary parabolic problem. In a second step, we pass to the limit as $R \to +\infty$, obtaining convergence to an entire solution of \eqref{system}. Assumption \eqref{cond on critical} enters in this second step in order to rule out the possibility that the limit profile $(u,v)$ is a pair with $v=u$, with $u$ possibly a constant. Roughly speaking, \eqref{cond on critical} makes the coexistence of $u$ and $v$ in the same region unfavorable with respect to the segregation, from the variational point of view. 

This kind of construction is inspired by \cite{BTWW, SZe, SZm}, where an analogue strategy was used to prove existence of solutions to
\begin{equation}\label{har sys}
\Delta u =  uv^2, \quad \Delta v = u^2 v, \quad u,v>0 \ \text{in $\R^N$}.
\end{equation}
With respect to \cite{BTWW, SZm, SZe}, however, the method has to be substantially modified. Solutions to \eqref{har sys} ``look like" harmonic function in the same way as solutions to \eqref{system} ``look like" solutions to the Allen-Cahn equation. Therefore, tools related with harmonic functions such as monotonicity formulae and blow-up analysis, which were crucially used in \cite{BTWW, SZm, SZe}, are not available in our context, and have to be replaced by a direct inspection of the variational background. In such an inspection it emerges the role of the competition parameter $\Lambda$, which is not present in \eqref{har sys} (of course, $\Lambda$ could be added in front of the coupling on the right hand side in \eqref{har sys}; but it could be absorbed with a scaling, and hence it would not play any role).

\begin{remark}
Let us consider the scalar equation $\Delta w + f(w) = 0$. The existence of a \emph{saddle-type} (or pizza) solution $w_k$ with the properties 
\begin{itemize}
\item[($i$)] $-M<w_k<M$ in $\R^2$; 
\item[($ii$)] $w_k \circ T_i= -w_k$ for every $i=1,\dots,k$, and $w_k(x_1,-x_2) = w_k(x_1,x_2)$ in $\R^2$; 
\item[($iii$)] $w_k>0$ in $\cS_k$. 
\end{itemize}
was established by Alessio, Calamai and Montecchiari in \cite{AlCaMo}, under slightly stronger assumption on $F$ with respect to those considered here. In some sense, $u_k-v_k$ looks like $w_k$, since they share the same symmetry properties, and for this reason we can call $(u_k,v_k)$ saddle-type (or pizza) solution.

Our method for Theorem \ref{thm: ex 2d} can be easily adapted also in the scalar case, giving an alternative proof for the existence result in \cite{AlCaMo}. For the sake of completeness, we present the details in the appendix of this paper. The main advantage is that our construction easily gives the following energy estimate
\begin{equation}\label{en est unk}
\int_{B_R} \left( \frac{1}{2} |\nabla w_k|^2 + F(w_k) \right) \le C R,
\end{equation}
with $C>0$ depending only on $k$, but not on $R$. Such estimate seems to be unknown, expect for the case $k=2$, where it was proved in \cite{CaTe}.
\end{remark}


Theorem \ref{thm: ex 2d} establishes the existence of infinitely many positive solutions to \eqref{system} in the plane. These can be regarded as solutions also in higher dimension $N \ge 3$, but it is natural to ask whether there exist solutions to \eqref{system} in $\R^N$ not coming from solutions in $\R^{N-1}$. We can give a positive answer to this question in any even dimension. Let $N=2m$, and let us consider the \emph{Simons cone}
\[
\cC= \left\{x \in \R^{2m}: \ x_1^2 + \cdots+x_m^2 = x_{m+1}^2 + \cdots + x_{2m}^2\right\}.
\]
We define two radial variables $s$ and $t$ by
\begin{equation}\label{rad var}
s= \sqrt{x_1^2 + \cdots + x_m^2} \ge 0, \qquad   t= \sqrt{x_{m+1}^2 + \cdots + x_{2m}^2} \ge 0.
\end{equation}
\begin{theorem}[Saddle solutions in $\R^{2m}$]\label{thm: ex saddle}
Let $m \ge 2$ be a positive integer, $p \ge 1$, $f \in C^{0,1}(\R)$ be odd, and suppose that its primitive $F$ satisfies \eqref{hp bistable}. Suppose moreover that \eqref{cond on critical} holds. Then, for every positive integer $m$, there exists a positive solution $(u,v)$ to system \eqref{system} in $\R^{2m}$ having the following properties:
\begin{itemize}
\item[($i$)] $0 < u,v<M$ in $\R^{2m}$; 
\item[($ii$)] $v(s,t)  = u(t,s)$;
\item[($iii$)] $u-v>0$ in $\cO =\{s>t\}$. 
\end{itemize}
\end{theorem}

Notice that $\{u-v=0\} = \cC$, and that $(u,v)$ looks like the saddle solution of the scalar Allen-Cahn equation found in \cite{CaTe}. The strategy of the proof is the same as the one of Theorem \ref{thm: ex saddle}. However, the proof of Theorem \ref{thm: ex saddle} is a bit simpler, since we can take advantage of an energy estimate like \eqref{en est unk}, which is known to hold for saddle solutions in $\R^{2m}$ (see formula (1.15) in \cite{CaTe}) but, as already observed, was unknown for saddle-type solutions in the plane.

\medskip

\noindent \textbf{Structure of the paper.} In Section \ref{sec:2d} we prove Theorem \ref{thm: ex 2d}. In Section \ref{sec:high} we prove Theorem \ref{thm: ex saddle}. In the appendix, we give an alternative proof with respect to \cite{AlCaMo} of the existence of saddle-type solutions for the scalar equation in the plane, yielding to the energy estimate \eqref{en est unk}.

\section{Saddle-type solutions for bistable systems in the plane}\label{sec:2d}

In this section we prove Theorem \ref{thm: ex 2d}. The existence of a solution in the whole plane $\R^2$ will be obtained by approximation with solutions in $B_R$.

Throughout this section, the positive integer $k$ (index of symmetry) will always be fixed, and hence the dependence of the quantities with respect to $k$ will often be omitted.

In the sector $\cS = \cS_k$, we define
\[
w_k = \min\left\{M, \frac{\alpha x_1- |x_2|}{\sqrt{2}}\right\},
\]
where $\alpha = \alpha_k = \tan\left(\pi/(2k)\right)$. Notice that $w_k > 0$ in $\cS_k$ and $w_k = 0$ on $\pa \cS_k$. Thus, we can extend $w_k$ in the whole of $\R^2$ by iterated odd reflections with respect to the lines $\ell_i$. In this way, we obtain a function, still denoted by $w_k$, defined in $\R^2$, with 
\begin{itemize}
\item[($i$)] $-M \le w_k  \le M$ in $\R^2$; 
\item[($ii$)] $w_k \circ T_i= -w_k$ for every $i=1,\dots,k$, and $w_k(x_1,-x_2) = w_k(x_1,x_2)$ in $\R^2$; 
\item[($iii$)] $w_k>0$ in $\cS_k$, 
\end{itemize}
that is, $w_k$ has the same symmetry properties of the saddle-type solutions in \cite{AlCaMo}.

Now, for any $\Omega \subset \R^2$ open, and for every $(u,v) \in H^1(\Omega, \R^2)$, we introduce the functional
\begin{equation}\label{def J}
J((u,v),\Omega):= \int_{\Omega} \left( \frac{1}{2} |\nabla u|^2 + \frac{1}{2} |\nabla v|^2 + W(u,v)  \right).
\end{equation}
Moreover, for $R>0$, we let $\cS_R = \cS_k \cap B_R$ and consider the set
\[
A_R:= \left\{ (u,v)\in H^1(B_R, \R^2)\left|\begin{array}{l} (u,v) = (w_k^+,w_k^-) \text{ on $\pa B_R$}, \ 0 \le u,v \le M \quad \text{in $B_R$} \\
v  = u \circ T_i \text{ for every $i=1,\dots,k$}, \\
u(x_1,-x_2) = u(x_1,x_2) \text{ in $B_R$}, \ u \ge v \quad \text{in $\cS_R$}, \ 
\end{array}\right.\right\}.
\]
Notice that $A_R \neq \emptyset$, since for instance $(w_k^+, w_k^-) \in A_R$.
\begin{lem}\label{lem: existence in balls}
For every $R>0$, there exists a solution $(u_R,v_R) \in A_R$ to 
\begin{equation}\label{k sys ball 2}
\begin{cases}
-\Delta u =f(u)-\Lambda  |u|^{p-1} u |v|^{p+1}  & \text{in}\quad B_R   \\
-\Delta v =f(v)-\Lambda  |u|^{p+1} |v|^{p-1} v & \text{in}\quad B_R \\
u = w_k^+, \quad v = w_k^- & \text{on $\pa B_R$}.
\end{cases}
\end{equation}
\end{lem}

\begin{proof}
The proof of the lemma is inspired by \cite[Theorem 4.1]{BTWW}. Since the weak convergence in $H^1$ implies the almost everywhere convergence, up to a subsequence, the set $A_R$ is weakly closed in $H^1$. Moreover, the functional $J(\cdot\,, B_R)$ is clearly bounded from below and weakly lower semi-continuous. Therefore, there exists a minimizer $(u_R,v_R)$ of $J(\cdot\,, B_R)$ in $A_R$. To show that such a minimizer is a solution to \eqref{k sys ball}, we consider the auxiliary parabolic problem
\begin{equation}\label{parabolic pb}
\begin{cases}
\pa_t U - \Delta U = \tilde f(U) - \Lambda |U|^{p-1} U |V|^{p+1} & \text{in $(0,+\infty) \times B_R$} \\
\pa_t V - \Delta V = \tilde f(V) - \Lambda |U|^{p+1}  |V|^{p-1} V & \text{in $(0,+\infty) \times B_R$} \\
U = w_k^+, \quad V=w_k^- & \text{on $(0,+\infty) \times \pa B_R$} \\
(U(0, \cdot), V(0,\cdot)) \in A_R,
\end{cases}
\end{equation}
where $\tilde f: \R \to \R$ is a globally Lipschitz continuous odd function such that $\tilde f(s) = f(s)$ for $s \in [-M-1,M+1]$. The existence and uniqueness of a local solution, defined on a maximal time interval $[0,T)$, follow by standard parabolic theory. Notice that
\[
\pa_t U - \Delta U = c_1(t,x) U,
\]
for
\[
c_1(t,x) = \begin{cases} - \Lambda |U(t,x)|^{p-1}  |V(t,x)|^{p+1} + \frac{\tilde f(U(t,x))}{U(t,x)}  & \text{if $U(t,x) \neq 0$} \\
0 & \text{if $U(t,x) = 0$}. \end{cases}
\]
Since $\tilde f(0)=0$, we have that $c_1$ is bounded from above by the Lipschitz constant $L$ of $\tilde f$, and it is not difficult to check that $U(t,\cdot) \ge 0$ in $B_R$ for every $t \in [0,T)$: indeed, taking into account the boundary conditions,
\begin{align*}
\frac{d}{dt} \left( \frac{1}{2} \int_{B_R} (U^-)^2 \right) &= - \int_{B_R} U^- (\pa_t U) = - \int_{B_R} U^- (\Delta U + c_1(t,x) U) \\
& \le  -\int_{B_R} |\nabla U^-|^2 + L \int_{B_R} (U^-)^2 \le L \int_{B_R} (U^-)^2,
\end{align*}
whence it follows that
\[
\frac{d}{dt} \left( e^{-2Lt} \int_{B_R} (U^-)^2 \right) \le 0.
\]
Therefore, the non-negativity of $U(t, \cdot)$ for $t \in (0,T)$ follows from the non-negativity of $U(0,\cdot)$. The same argument also shows that $V(t,\cdot) \ge 0$ for every such $t$. Using the positivity of $U$, it is not difficult to prove that $U$ is also uniformly bounded from above: since $\tilde f(M) = 0$, we have 
\[
\pa_t (M-U) - \Delta (M-U) \ge c_2(t,x) (M-U), 
\]
where $c_2$ is the bounded function
\[
c_2(t,x) = \begin{cases} \frac{\tilde f(U(t,x))-\tilde f(M)}{U(t,x)-M}  & \text{if $U(t,x) \neq 0$} \\
0 & \text{if $U(t,x) = M$}, \end{cases}
\]
and the same argument used above implies that $0 \le U \le M$ on $(0,T) \times B_R$. Similarly, $0 \le V \le M$. As a consequence, the solution $(U,V)$ can be globally continued in time on $(0,+\infty)$. Furthermore, in \eqref{parabolic pb} we can replace $\tilde f$ with $f$, since they coincide on $[-M-1,M+1]$.

We also observe that, since $U$ is constant in time on $\pa B_R$, the energy of the solution is non-increasing:
\begin{equation}\label{diss}
\begin{split}
\frac{d}{dt} J( (U(t,\cdot), V(t,\cdot)); B_R) & = \int_{B_R}   \nabla U \cdot \nabla U_t + \nabla V \cdot \nabla V_t  + \pa_1 W(U,V) U_{t} + \pa_2 W(U,V) V_t  \\
& = \int_{B_R}  \left( -\Delta U + \pa_1 W(U,V)\right) U_{t} +  \left( -\Delta V + \pa_2 W(U,V)\right) V_{t}  \\
& = -  \int_{B_R} U_{t}^2 + V_t^2 \le 0.
\end{split}
\end{equation}

As in \cite{BTWW}, we can now show that $A_R$ is positively invariant under the parabolic flow. Let $(U,V)$ be a solution with initial datum in $A_R$. By the symmetry of \eqref{parabolic pb}, we have that $(V(t, T_i x), U(t,T_i x))$ is still a solution. By the symmetry of initial and boundary data, and by uniqueness, such solution must coincide with $(U(t,\cdot), V(t,\cdot))$. This means in particular that $V(t,x) = U(t, T_i x)$. Likewise, $U(t,x_1,-x_2) = U(t, x_1,x_2)$. Notice that the symmetries imply that $U-V = 0$ on $\pa \cS_k$. Thus, recalling that $\cS_R = \cS_k \cap B_R$, we have
\begin{equation}\label{pb in S}
\begin{cases}
\pa_t (U-V) - \Delta (U-V) = c(t,x) (U-V) & \text{in $(0,+\infty) \times \cS_R$} \\
U-V \ge 0 & \text{on $(0,+\infty) \times \pa \cS_R$} \\
U-V \ge 0 & \text{on $\{0\} \times \cS_R$},
\end{cases}
\end{equation}
where the bounded function $c$ is defined by
\[
c(t,x) = \begin{cases}  
\frac{f(U(t,x))-f(V(t,x))}{U(t,x)- V(t,x)} + \Lambda U^p(t,x) V^p(t,x) & \text{if } U(t,x) \neq V(t,x) \\
\Lambda U^p(t,x) V^p(t,x) & \text{if } U(t,x) = V(t,x).
\end{cases}
\]
The parabolic maximum principle implies that $U \ge V$ in $\cS_R$ globally in time, and, in turn, this gives the invariance of $A_R$.

At this point we consider the solution $(U_R, V_R)$ to \eqref{parabolic pb} with initial datum $(u_R, v_R)$, minimizer of $J(\cdot\,, B_R)$ in $A_R$. By minimality in $A_R$ and by \eqref{diss}, we have 
\[
J((u_R, v_R); B_R) \le J ((U_R(t,\cdot), V_R(t,\cdot)); B_R) \le J((u_R, v_R); B_R) \quad \implies \quad U_{t}^2 + V_t^2 \equiv 0,
\]
and hence $U_R \equiv u_R$ and $V_R \equiv v_R$. But then $(u_R(x), v_R(x))$ is a (stationary) solution of the parabolic problem \eqref{parabolic pb}, that is, it solves the stationary problem \eqref{k sys ball 2}, and in addition $(u_R, v_R) \in A_R$. This completes the proof.
\end{proof}

We are ready to complete the:

\begin{proof}[Proof of Theorem \ref{thm: ex 2d}]
First, of all, we discuss the convergence of $\{(u_R, v_R): R >1\}$. Let $\rho >1$. Since $0 \le u_R, v_R \le M$, we have that 
\[
|\Delta u_R(x)|  \le \max_{s \in [0,M]} |f(s)| + \Lambda M^{2p+1}, \quad |\Delta v_R(x)|  \le \max_{s \in [0,M]} |f(s)| + \Lambda M^{2p+1}.
\]
Thus interior $L^p$ estimates (see e.g. \cite[Chapter 9]{GiTr}), applied in balls of radius $2$ centered in points of $\overline{B_\rho}$ with $p>N$, and the Morrey embedding theorem, imply that there exists $C>0$ depending only on $M$ and $\Lambda$ (but independent of $R$ and $\rho$) such that 
\begin{equation}\label{grad est}
\|u_R\|_{C^{1,\alpha}(\overline{B_\rho})} + \|v_R\|_{C^{1,\alpha}(\overline{B_\rho})}  \le C \qquad \text{in $B_\rho$, for all $R>\rho+2$}
\end{equation}
(for every $0<\alpha<1$). By the Ascoli-Arzel\`a theorem, up to a subsequence $\{(u_R, v_R)\}$ converges in $C^{1,\alpha}(\overline{B_\rho})$ to a solution in $B_\rho$, for every $0<\alpha<1$. Taking a sequence $\rho \to +\infty$, a diagonal selection finally gives $(u_R, v_R) \to (u,v)$ in $C^{1,\alpha}_{\loc}(\R^2)$, up to a subsequence, with $(u,v)$ solution to
\[
\begin{cases}
-\Delta u = f(u) -\Lambda u^p v^{p+1} & \text{in $\R^2$} \\ 
-\Delta v = f(v) -\Lambda u^{p+1} v^{p} & \text{in $\R^2$} \\ 
0 \le u,v \le M & \text{in $\R^2$}.
\end{cases}
\]
Notice that, by convergence, $(u,v)$ satisfies the symmetry property ($ii$) in Theorem \ref{thm: ex 2d}, and moreover $u-v \ge 0$ in $\cS_k$. As in \eqref{pb in S}, for any $\rho>0$
\[
\begin{cases}
-\Delta (u-v) = c(x) (u-v) & \text{in $\cS_\rho$} \\
u-v \ge 0 & \text{in $\cS_\rho$},
\end{cases}
\]
for a bounded function $c$. Thus, the strong maximum principle implies that either $u > v$ in $\cS_\rho$, of $u \equiv v$ in $\cS_\rho$. Since $\rho>0$ is arbitrarily chosen, we have that either $u > v$ in $\cS$, of $u \equiv v$ in $\cS$. We claim that the latter alternative cannot take place. To prove this claim, we use a comparison argument similar as the one by Cabr\'e and Terra in \cite{CaTe} for the construction of the saddle solution for scalar bystable equations. First of all, we observe that, by symmetry, any function in $A_R$ is determined only by its restriction on $\cS_k$. Thus the minimality of $(u_R, v_R)$ can be read as
\[
J((u_R,v_R), \cS_R) \le J((u,v), \cS_R) \qquad \forall (u,v) \in A_R.
\]
Let $1<\rho<R-2$, and let $\xi$ be a radial smooth cut-off function with $\xi \equiv 1$ in $B_{\rho-1}$, $\xi \equiv 0$ in $B_\rho^c$, $0 \le \xi \le 1$. We define
\[
\varphi_R(x) = \xi (x) w_k^+(x)   +   (1-\xi(x)) u_R(x) = \xi (x) \min\left\{M, \frac{\alpha x_1- |x_2|}{\sqrt{2}}\right\} +   (1-\xi(x)) u_R(x),
\]
and
\[ 
\psi_R(x) = \xi(x) w_k^-(x) +(1-\xi (x)) v_R(x) = (1-\xi (x)) v_R(x),
\]
where we recall that $\alpha= \tan\left(\pi/(2k)\right)$. It is immediate to verify that $(\varphi_R,\psi_R)$ is an admissible competitor for $(u_R, v_R)$ on $\cS_R$. Moreover, by \eqref{grad est}, there exists $C>0$ such that
\begin{equation}\label{gradient bound 2}
\|\varphi_R\|_{W^{1,\infty}(B_{\rho})} + \|\psi_R\|_{W^{1,\infty}(B_{\rho})} \le C \qquad \forall R >\rho+2.
\end{equation}
By minimality
\[
J((u_R, v_R), \cS_R) \le J((\varphi_R,\psi_R), \cS_R),
\]
and since $(\varphi_R, \psi_R) = (u_R, v_R)$ in $\cS_R \setminus \cS_\rho$ we deduce that
\begin{equation}\label{1561}
J((u_R, v_R), \cS_\rho) \le J((\varphi_R,\psi_R), \cS_\rho) \le J((\varphi_R, 0), \cS_{\rho-1}) + C |\cS_{\rho} \setminus \cS_{\rho-1}| 
\end{equation}
where we used the global boundedness of $\{(\varphi_R, \psi_R)\}$ in $W^{1,\infty}(B_\rho)$, see \eqref{gradient bound 2}. The last term can be easily computed as
\[
|\cS_{\rho} \setminus \cS_{\rho -1}| = \frac{\pi}{k} ( \rho^2 - (\rho-1)^2)  \le \frac{2\pi}{k} \rho.
\]
For the first term, recalling that $F(M) = 0$, $\xi \equiv 1$ in $B_{\rho-1}$, and $w_k>0$ in $\cS_k$, we have
\begin{align*}
\int_{\cS_{\rho-1}} & \left( \frac{1}{2} |\nabla \varphi_R|^2 + F(\varphi_R) + F(0) \right) \\
& \qquad = \int_{\cS_{\rho-1} \cap \left\{\alpha x_1-|x_2| < \sqrt{2} M \right\}} \left( \frac{1}{2} |\nabla w_k|^2 + F(w_k)\right) + \int_{\cS_{\rho-1}} F(0) \\
& \qquad \le C |\cS_{\rho-1} \cap \left\{\alpha x_1-|x_2| < \sqrt{2} M \right\}| + F(0) |\cS_{\rho-1}|. 
\end{align*}
The set
\[
\cS_k \cap \left\{\alpha x_1-|x_2| < \sqrt{2} M \right\} 
\]
is contained in the (non-disjoint) union of the two strips
\[
\left\{\alpha x_1 - \sqrt{2}M < x_2 < \alpha x_1, \ x_1>0\right\}  \cup \left\{-\alpha x_1 < x_2 < 2 \sqrt{M} -\alpha x_1, \ x_1>0\right\} =S_1 \cup S_2.
\]
Therefore, 
\begin{align*}
 |\cS_{\rho-1} \cap \left\{\alpha x_1-|x_2| < \sqrt{2} M \right\}| &\le  |S_1 \cap \{0<x_1<\rho\}| + |S_2 \cap \{0<x_1<\rho\}| \\  
 & = 2 \int_0^\rho \left( \int_{-\sqrt{2} M + \alpha x_1}^{\alpha x_1} 1 \, dx_2 \right) dx_1 = 2 \sqrt{2} M\rho.
\end{align*}
Coming back to \eqref{1561}, we conclude that there exists a constant $C>0$ such that, for every $\rho >1$ and $R>\rho+2$,
\[
J((u_R, v_R), \cS_\rho)  \le C \rho + F(0) |\cS_{\rho-1}| 
\]
for every $1<\rho<R-2$, where $C>0$ is a positive constant independent of both $\rho$ and $R$. Passing to the limit as $R \to +\infty$, we infer by $C^1_{\loc}$-convergence that
\begin{equation}\label{energy estimate}
J((u,v), \cS_\rho) \le C \rho + F(0) |\cS_{\rho-1}| 
\end{equation}
for every $\rho>1$. Notice that, in this estimate, the leading term as $\rho \to +\infty$ is 
\[
F(0) |\cS_{\rho-1}| \sim \frac{\pi}{k} F(0) \rho^2.
\]
Suppose now by contradiction that $u \equiv v$ in $\cS_k$. Recalling that $0\le u,v \le M$, we have that 
\begin{align*}
J((u,v), \cS_\rho) &= \int_{\cS_\rho} |\nabla u|^2 + W(u,u) \ge \int_{\cS_\rho} \min_{s \in [0,M]} W(s,s) \\
& = \min_{s \in [0,M]} W(s,s) |\cS_{\rho}| \sim \frac{\pi}k \min_{s \in [0,M]} W(s,s) \rho^2
\end{align*}
as $\rho \to +\infty$. Comparing with \eqref{energy estimate}, we obtain a contradiction for large $\rho$, thanks to assumption \eqref{cond on critical}. Therefore, $u >v$ in $\cS_k$. Since $u = v$ on $\pa \cS_k$, we also infer that both $u$ and $v$ cannot be constant. The maximum principle implies then that $u, v >0$ in $\R^2$, and from this it is not difficult to deduce that $u, v <M$: indeed, if $u(x_0) = M$, then $x_0$ is a strict maximum point for $u$ with 
\[
\Delta u(x_0) = -f(M) + \Lambda M^p v(x_0)^{p+1} = \Lambda M^p v(x_0)^{p+1} >0,
\]
which is not possible. This completes the proof.
\end{proof}

\section{Existence of saddle solutions in higher dimension}\label{sec:high}

The proof of Theorem \ref{thm: ex saddle} follows the same strategy as the one of Theorem \ref{thm: ex 2d}, being actually a bit simpler. Let $m \ge 2$ be a positive integer. By \cite[Theorem 1.3]{CaTe}\footnote{For the existence and the energy estimate in the theorem, it is sufficient that $f$ is locally Lipschitz, rather than $C^1$}, under our assumption \eqref{hp bistable} on $F$ the Allen-Cahn equation $\Delta w + f(w) = 0$ in $\R^{2m}$ admits a saddle solution $w_m$, that is a solution satisfying: 
\begin{itemize}
\item[($i$)] $w$ depends only on the variables $s$ and $t$ defined in \eqref{rad var}; 
\item[($ii$)] $w_m(s,t)  = -w_m(t,s)$;
\item[($iii$)] $w_m>0$ in $\cO =\{s>t\}$. 
\end{itemize}
In addition, $|w_m|<M$ in $\R^{2m}$, and
\begin{equation}\label{E saddle}
\int_{B_R} \frac{1}{2} |\nabla w_m|^2 + F(w_m) \le C R^{2m-1} \qquad \text{for all $R>1$},
\end{equation} 

Now, as in Section \ref{sec:2d}, we consider the energy functional $J((u,v), \Omega)$ defined in \eqref{def J} (in this section $\Omega \subset \R^{2m}$), and the set
\[
A_R:= \left\{ (u,v)\in H^1(B_R, \R^2)\left|\begin{array}{l} (u,v) = (w_m^+,w_m^-) \text{ on $\pa B_R$}, \\
v(s,t)  = u(t,s),\\
u \ge v \quad \text{in $\cO_R$}, \ 0 \le u,v \le M \quad \text{in $B_R$}\\
\end{array}\right.\right\},
\]
where $\cO_R = \cO \cap B_R$.
\begin{lem}\label{lem: existence in balls m}
For every $R>0$, there exists a solution $(u_R,v_R) \in A_R$ to 
\begin{equation}\label{k sys ball}
\begin{cases}
-\Delta u =f(u)-\Lambda  |u|^{p-1} u |v|^{p+1}  & \text{in}\quad B_R   \\
-\Delta v =f(v)-\Lambda  |u|^{p+1} |v|^{p-1} v & \text{in}\quad B_R \\
u = w_m^+, \quad v = w_m^- & \text{on $\pa B_R$}.
\end{cases}
\end{equation}
\end{lem}
The proof is analogue to the one of Lemma \ref{lem: existence in balls}, and is omitted.

\begin{proof}[Proof of Theorem \ref{thm: ex saddle}]
As in the $2$-dimensional case, we can prove that up to a subsequence $(u_R, v_R) \to (u,v)$ in $C^{1,\alpha}_{\loc}(\R^2)$ as $R \to +\infty$, with $(u,v)$ solution to
\[
\begin{cases}
-\Delta u = f(u) -\Lambda u^p v^{p+1} \\ 
-\Delta v = f(v) -\Lambda u^{p+1} v^{p} \\
0 \le u,v \le M
\end{cases} \quad \text{in $\R^{2m}$}.
\]
By convergence, $v(s,t) = u(t,s)$, $u-v \ge 0$ in $\cO$, and $0 \le u,v \le M$ in $\R^{2m}$. Also, for every $\rho>0$
\[
\begin{cases}
-\Delta (u-v) = c(x) (u-v) & \text{in $\cO_\rho$} \\
u-v \ge 0 & \text{in $\cO_\rho$},
\end{cases}
\]
for a bounded function $c$. Thus, the strong maximum principle implies that either $u > v$ in $\cO$, of $u \equiv v$ in $\cO$. We claim that the latter alternative cannot take place. Let $1<\rho<R-2$, and let $\xi$ be a radial smooth cut-off function with $\xi \equiv 1$ in $B_{\rho-1}$, $\xi \equiv 0$ in $B_\rho^c$, $0 \le \xi \le 1$. We define 
\[
\varphi_R = \xi w_m^+ + (1-\xi) u_R, \quad \psi_R = \xi w_m^- + (1-\xi) v_R.
\]
This is an admissible competitor in $A_R$, which coincides with $(u_R, v_R)$ on $B_R \setminus B_\rho$. Therefore, by minimality and recalling \eqref{E saddle}, we have
\begin{equation}
\begin{split}
J((u_R, v_R), B_\rho) &\le J((w_m^+,w_m^-), B_{\rho-1}) + C |B_{\rho} \setminus  B_{\rho-1}| \\
& \le E(w_m, B_{\rho-1}) + \int_{B_{\rho-1}} F(0) + C \rho^{2m-1} \le C \rho^{2m-1} + F(0) |B_{\rho-1}|.
\end{split}
\end{equation}
If, by contradiction, $u \equiv v$ in $\cO$, then we have that 
\[
J((u,v), B_\rho) = \int_{B_\rho} |\nabla u|^2 + W(u,u) \ge \int_{B_\rho} \min_{\sigma \in [0,M]} W(\sigma,\sigma) = \min_{\sigma \in [0,M]} W(\sigma,\sigma) |B_{\rho}|. 
\]
Comparing with \eqref{energy estimate}, we obtain a contradiction for large $\rho$, thanks to assumption \eqref{cond on critical}. Thus, $u>v$ in $\cO$, and the conclusion of the proof is straightforward.
\end{proof}

\appendix

\section{Alternative construction of saddle-type planar solutions}

In this appendix we consider the scalar Allen-Cahn equation
\begin{equation}\label{scalar}
-\Delta w   = f(w) \qquad \text{in $\R^2$},
\end{equation}
and we prove the following result:
\begin{theorem}\label{thm: ex scal}
Let $f \in C^{0,\alpha}(\R)$ be odd, and suppose that its anti-primitive $F$ satisfies \eqref{hp bistable}. Then, for every $k \in \N$, there exists a solution $w_k$ having the following properties:
\begin{itemize}
\item[($i$)] $-M<w_k<M$ in $\R^2$; 
\item[($ii$)] $w_k \circ T_i= -w_k$ for every $i=1,\dots,k$, and $w_k(x_1,-x_2) = w_k(x_1,x_2)$ in $\R^2$; 
\item[($iii$)] $w_k>0$ in $\cS_k$. 
\end{itemize}
Moreover, there exists a constant $C>0$ (possibly depending on $k$) such that
\begin{equation}\label{energy estimate scalar}
\int_{B_R} \left( \frac12 |\nabla w|^2 + F(w) \right) \le C R \qquad \text{for every $R>0$}.
\end{equation}
\end{theorem}

\begin{remark}
The existence of a solution $w_k$ with the properties ($i$)-($iii$) was established by Alessio, Calamai and Montecchiari in \cite{AlCaMo}. In \cite{AlCaMo} the authors also obtained a more precise description of the asymptotic behavior of $w_k$ at infinity. On the other hand, the validity of the estimate \eqref{energy estimate scalar} was unknown.

In order to show that $w_k$ fulfills \eqref{energy estimate scalar}, we provide an alternative existence proof with respect to the one in \cite{AlCaMo}. It is tempting to conjecture that the solutions given by Theorem \ref{thm: ex scal}, and those found in \cite{AlCaMo}, coincide.

Our alternative proof is strongly inspired by \cite{CaTe}, where Cabr\'e and Terra proved existence of solutions in $\R^{2m}$ to \eqref{scalar} vanishing on the Simon's cone (when restricted to the case $m=1$ - i.e., when we consider \eqref{scalar} in the plane - their result establishes the existence of the solution $w_2$). We first prove the existence of a solution $w_R=w_{k,R}$ to \eqref{scalar} in the ball $B_R$, for every $R>0$, by variational argument. Passing in a suitable way to the limit as $R \to +\infty$, we shall obtain a solution in the whole plane $\R^2$ having the desired energy estimate. 

The main simplification with respect to the proof of Theorem \ref{thm: ex 2d} stays in the fact that, dealing with a single equation, we will not need an auxiliary parabolic problem, but we will be able to prove the existence of a solution in $B_R$ with the desired symmetry properties directly by variational methods.
\end{remark}

The proof of Theorem \ref{thm: ex scal} takes the rest of this appendix. Let us fix $k$. For any $\Omega \subset \R^2$ open, and for every $w \in H^1(\Omega)$, we define 
\[
E(w,\Omega):= \int_{\Omega} \left( \frac{1}{2} |\nabla w|^2 + F(w) \right).
\]
For $R>0$, we consider $\cS_{R} := B_R \cap \cS_k$ and the set
\[
H_R:= \left\{ w \in H_0^1(S_R): w(x_1,-x_2) = w(x_1,x_2) \ \text{a.e. in $\cS_R$}\right\}.
\]

\begin{lem}\label{lem: sc 1}
For every $R>0$, problem 
\begin{equation}\label{aux 2}
\begin{cases}
-\Delta w = f(w) & \text{in $B_R$} \\
w(x_1,-x_2) = w(x_1,x_2) & \text{in $B_R$} \\
w = 0 & \text{on $\pa B_R$},
\end{cases}
\end{equation}
has a solution $w_R$, satisfying ($ii$) in Theorem \ref{thm: ex scal}. Moreover, $- M \le w_R \le M$ in $B_R$, $w_R \ge 0$ in $\cS_R$, and 
\[
E(w_R, S_R) = \min\{E(w,S_R): \ w \in H_R\}.
\]
\end{lem}

\begin{proof}
At first, we search a solution to the auxiliary problem
\begin{equation}\label{aux 1}
\begin{cases}
-\Delta w = f(w) & \text{in $\cS_R$} \\
w \in H_0^1(S_R), \quad w \ge 0 & \text{in $\cS_R$} \\
w(x_1,x_2) = w(x_1,-x_2) & \text{in $\cS_R$},
\end{cases}
\end{equation}
by minimizing the function $E(w,\cS_R)$ in $H$. The existence of a minimizer follows easily by the direct method of the calculus of variations. Since $E(w,\cS_R) = E(|w|,\cS_R)$, it is not restrictive to suppose that $w_R \ge 0$. Also, since $E( \min\{w,M\},\cS_R) \le E(w,\cS_R)$ by assumption \eqref{hp bistable}, we can suppose that $w_R \le M$. Clearly, $w_R$ solves the first equation in \eqref{aux 1} in the set $\cS_R \setminus \{\theta=0\}$. The fact that $w_R$ is also a solution across $\cS_R \cap \{\theta=0\}$ (thus a solution in $\cS_R$) follows by the principle of symmetric criticality (see e.g. \cite[Theorem 1.28]{Wi} for a simple proof of this result, sufficient to our purposes). 

Notice that $w_R \in C^1(\{0<r<R, -\pi/2k \le \theta \le \pi/2k\})$ by standard elliptic regularity. Thus, we can reflect $w_R$ $2k$ times in an odd way across $\ell_1,\dots,\ell_k$, obtaining a solution in $B_R \setminus \{0\}$. To see that $w_R$ is in fact a solution in $B_R$, we take a smooth function $\eta_\delta \in C^\infty(\overline{B_R})$ with $\eta_\delta  \equiv 0$ in $B_\delta$, $\eta_\delta \equiv 1$ in $B_{2\delta} \setminus B_\delta$, and $|\nabla \eta_\delta| \le C/\delta$ in $B_R$. Then, for every $\varphi \in C^\infty_c(B_R)$, we have
\[
\int_{B_R} \nabla w_R \cdot \nabla (\varphi \eta_\delta) - \int_{B_R} f(w_R) \varphi \eta_\delta = 0,
\]
since $\varphi \eta_\delta$ is an admissible test function in $B_R \setminus \{0\}$. Passing to the limit as $\delta \to 0^+$, we deduce that 
\[
\int_{B_R} \nabla w_R \cdot \nabla \varphi  - \int_{B_R} f(w_R) \varphi  = 0 \qquad \forall \varphi \in C^\infty_c(B_R),
\]
that is, $w_R$ is a weak solution to \eqref{scalar} in $B_R$. This completes the proof.
\end{proof}

\begin{proof}[Proof of Theorem \ref{thm: ex scal}]
We wish to pass to the limit as $R \to +\infty$ and obtain a solution in the whole plane $\R^2$ as limit of the family $\{w_R\}$. As in the previous sections, by elliptic estimates we have that, up to a subsequence $w_R \to w$ in $C^{1,\alpha}_{\loc}(\R^2)$ as $R \to \infty$, for every $\alpha \in (0,1)$.
The limit $w$ inherits by $w_R$ the symmetry property ($ii$) in Theorem \ref{thm: ex scal}. Moreover, $w \ge 0$ in the sector $\cS=\cS_k$, and $|w| \le M$ in the whole plane $\R^2$. 
Actually, the strict inequality $|w|<M$ holds, by the strong maximum principle. To complete the proof of the theorem, it remains then to show that $w$ satisfies estimate \eqref{energy estimate scalar}, and that $w \not \equiv 0$. 

As in the proof of Theorem \ref{thm: ex 2d}, $\{w_R\}$ has a uniform gradient bound: there exists $C>0$ (independent of $R$) such that
\begin{equation}\label{gradient bound}
\|\nabla w_R\|_{L^\infty(B_{R-1})} \le C \qquad \forall R >1.
\end{equation}
For an arbitrary $\rho>1$, let now $R>\rho + 2$, and let $\xi \in C^\infty_c(B_\rho)$, with $\xi \equiv 1$ in $B_{\rho -1}$. We consider the following competitor for $w_R$: 
\[
\varphi_R(x) = \xi(x) \min\left\{ \frac{\alpha x_1- |x_2|}{\sqrt{2}}, M \right\} +   (1-\xi(x)) w_R(x).
\]
Notice that this is the same type of competitor we used in Theorem \ref{thm: ex 2d}.
%
%
%
By minimality
\[
E(w_R, S_R) \le E(\varphi_R, S_R),
\]
and since $w_R = \varphi_R$ in $\cS_R \setminus \cS_\rho$ we deduce that
\begin{equation}\label{1561s}
\begin{split}
\int_{\cS_\rho} \left( \frac{1}{2} |\nabla w_R|^2 + F(w_R) \right) &\le \int_{\cS_\rho} \left( \frac{1}{2} |\nabla \varphi_R|^2 + F(\varphi_R) \right) \\
 & \le \int_{\cS_{\rho-1}} \left( \frac{1}{2} |\nabla \varphi_R|^2 + F(\varphi_R) \right) + C |\cS_{\rho} \setminus \cS_{\rho -1}|,
\end{split}
\end{equation}
where we used the global boundedness of $\{\varphi_R\}$ in $W^{1,\infty}(B_\rho)$, see \eqref{gradient bound}. At this point we can proceed as in the conclusion of the proof of Theorem \ref{thm: ex 2d}: the right hand side in \eqref{1561s} can be estimated by $C \rho$, with $C$ independent of $\rho$. Thus, we conclude that there exists a constant $C>0$ such that, for every $\rho >1$ and $R>\rho+2$,
\[
E(w_R, \cS_\rho) = \int_{S_\rho} \left( \frac{1}{2} |\nabla w_R|^2 + F(w_R) \right) \le C \rho,
\]
Passing to the limit as $R \to +\infty$, we infer by $C^1_{\loc}$-convergence that
\[
E(w, \cS_\rho) \le C \rho,
\]
which implies, by symmetry, the estimate \eqref{energy estimate scalar}.

Suppose finally that $w \equiv 0$. Then the energy estimate \eqref{energy estimate scalar} would give for every $\rho>1$
\[
\pi F(0) \rho^2 = E(0, B_\rho) \le C \rho,
\]
which is not possible if $\rho$ is sufficiently large. This proves that $w \not \equiv 0$, and completes the proof of Theorem \ref{thm: ex scal}.
\end{proof}


\end{document}